\documentclass[fleq]{amsart}
\usepackage{amssymb,amsmath,latexsym,amsbsy,color,enumerate}
\numberwithin{equation}{section}
\def\NN{\mbox{$I\hspace{-.06in}N$}}

\def\CC{\mbox{$C\hspace{-.11in}\protect\raisebox{.5ex}{\tiny$/$}
\hspace{.06in}$}}

\newtheorem{theorem}{Theorem}
\newtheorem{definition}[theorem]{Definition}
\newtheorem{lemma}{Lemma}
\newtheorem{remark}{Remark}
\newtheorem{proposition}{Proposition}
\newtheorem{corollary}{Corollary}
\newtheorem{example}{Example}
\begin{document}
 \title{Sets non-thin at $\infty$ in $\mathbb C ^m$, and the growth of sequences of entire functions of genus zero}
\author{Tuyen Trung Truong}
    \address{Department of Mathematics, Indiana University Bloomington, IN 47405 USA}
 \email{truongt@indiana.edu}
\thanks{The author would like to thank Professor Bedford and Professor Levenberg for their generous and fruitful help.}
    \date{\today}
    \begin{abstract}
In this paper we define the notion of non-thin at $\infty$ as follows: Let $E$ be a subset of $\mathbb C^m$. For any $R>0$ define $E_R=E\cap \{z\in
\mathbb C ^m :|z|\leq R\}$. We say that $E$ is non-thin at $\infty$ if
\begin{eqnarray*}
\lim _{R\rightarrow\infty}V_{E_R}(z)=0
\end{eqnarray*}
for all $z\in \mathbb C^m$, where $V_E$ is the pluricomplex Green function of $E$.

This definition of non-thinness at $\infty$ has good properties: If $E\subset \mathbb C^m$ is non-thin at $\infty$ and $A$ is pluripolar then
$E\backslash A$ is non-thin at $\infty$, if $E\subset \mathbb C^m$ and $F\subset \mathbb C^n$ are closed sets non-thin at $\infty$ then $E\times F\subset
\mathbb C^m\times \mathbb C^n$ is non-thin at $\infty$ (see Lemma \ref{Lem1}).

Then we explore the properties of non-thin at $\infty$ sets and apply this to extend the results in \cite{mul-yav} and \cite{trong-tuyen}.
 \end{abstract}
\maketitle
\section{Introduction}       % Enter section title between curly braces
\label{SecIntroduction}

Fix $m\in \NN$, let $\mathbb C ^m$ be the usual complex $m$-dimensional complex space. Before going into the main points, we recall some facts about the
potential theory in $\mathbb C^m$. Let $U$ be an open subset of $\mathbb C^m$. A function $u:U\rightarrow [-\infty ,\infty )$ (see \cite{klimek}) is
called PSH in $U$ (written $u\in PSH(U)$) if $u$ is upper-semicontinuous and when restricted to any complex line $L\simeq \mathbb C$ then $u$ is
subharmonic (see \cite{klimek}).

A $PSH(\mathbb C^m)$ function $u$ is said to be of minimal growth if
\begin{eqnarray*}
u(z)-\log |z|\leq O(1), \mbox{ as }|z|\rightarrow \infty ,
\end{eqnarray*}
here $|z|$ is the usual Euclide norm of an element $z\in \mathbb C^m$. We denote the set of all such functions by $\mathbb L$.

Let $E$ be a subset in $\mathbb C^m$. Then the pluricomplex Green function of the set $E$ with pole at infinity (see \cite{klimek}, page 184) is
\begin{eqnarray*}
V_E(z)=\sup \{u(z):~u\in \mathbb L,~u|_{E}\leq 0\}~(z\in \mathbb C ^m).
\end{eqnarray*}
$V_E$ is also called the Siciak extremal function of the set $E$. By definition we see that $V_E\geq 0$.

For any function $u:\mathbb C ^m\rightarrow [-\infty ,\infty )$ we define its upper-semicontinuos regularization $u^*$ (see \cite{klimek}) by
\begin{eqnarray*}
u^*(z)=\limsup _{\zeta \rightarrow z}u(\zeta ).
\end{eqnarray*}
Let $V_E^*(z)$ be the regularization of the Green function $V_E(z)$. It is well-known (see \cite{klimek}) that there are only three cases:

Case 1: $V_E^*(z)\equiv +\infty$. Then $E$ is pluripolar.

Case 2: $V_E^*(z)\in \mathbb L$, and $V_E^*\not\equiv 0$.

Case 3: $V_E^*(z)\equiv 0$.

When $m=1$ then $V_E^*(z)\equiv 0$ iff the set $E$ is non-thin at $\infty$, or equivalently the set $E^*=\{z:~1/z\in E \}$ is non-thin at $0$ (see
\cite{mul-yav}, \cite{cegrell-kolodziej-levenberg}). Recall that a subset $E$ of $\mathbb C^m$ is pluri-thin (or thin for brevity) at a point $a\in
\mathbb C ^m$ if either $a$ is not a limit point of $E$ or there is a neighborhood $U$ of a and a function $u\in PSH(U)$ such that
\begin{eqnarray*}
\limsup _{z\rightarrow a,~z\in E\backslash a}u(z)<u(a).
\end{eqnarray*}
If $E$ is not thin then it is called non-thin.

In \cite{mul-yav} the authors proved the following result (see also \cite{cegrell-kolodziej-levenberg}, page 270)
\begin{proposition}
Let $E$ be a closed subset of $\mathbb C$. Then the following four statements are equivalent

1. $E$ is non-thin at $\infty$.

2. For any $z\in \mathbb C$ we have
\begin{eqnarray*}
\lim _{R\rightarrow\infty} V_{E_R}(z)=0,
\end{eqnarray*}
where we define $E_R=E\cap \{z\in \mathbb C :~|z|\leq R\}$.

3. If sequences $(P_n)$ of polynomials and $k_n\geq deg(P_n)$ satisfy
\begin{eqnarray*}
\limsup _{n\rightarrow\infty}|P_n(z)|^{1/k_n}\leq 1
\end{eqnarray*}
for all $z\in E$, then
\begin{eqnarray*}
\limsup _{n\rightarrow\infty}|P_n(z)|^{1/k_n}\leq 1
\end{eqnarray*}
for all $z\in \mathbb C$.

4. $V_E(z)\equiv 0$. \label{Pro1}\end{proposition} When $m>1$, as to our knowledge, there was no conception of non-thin at $\infty$. The difficulty here
is because we can not reduce the definition of non-thinness of $E$ at $\infty$ to the non-thinness at $0$ of some other sets $E^*$, as it was when $m=1$.
However, Proposition \ref{Pro1} suggests a way to define non-thinness at $\infty$ in higher dimensions:
\begin{definition}
Let $E$ be a subset of $\mathbb C^m$. For any $R>0$ define $E_R=E\cap \{z\in \mathbb C^m :|z|\leq R\}$. We say that $E$ is non-thin at $\infty$ if
\begin{eqnarray*}
\lim _{R\rightarrow\infty}V_{E_R}(z)=0
\end{eqnarray*}
for all $z\in \mathbb C^m.$

If $E$ is not non-thin at $\infty$ then we say $E$ is thin at $\infty$. \label{Def1}\end{definition} This definition of non-thinness at $\infty$ has good
properties: If $E\subset \mathbb C^m$ is non-thin at $\infty$ and $A$ is pluripolar then $E\backslash A$ is non-thin at $\infty$, if $E\subset \mathbb
C^m$ and $F\subset \mathbb C^n$ are closed sets non-thin at $\infty$ then $E\times F\subset \mathbb C^m\times \mathbb C^n$ is non-thin at $\infty$ (see
Lemma \ref{Lem1}).

Our first result in this paper is the following, which is an analog to the case $m=1$:
\begin{theorem}
Let $E$ be a closed subset of $\mathbb C^m$. Then the following two statments are equivalent:

1. $E$ is non-thin at $\infty$.

2. If sequences $(P_n)$ of polynomials and $k_n\geq deg(P_n)$ satisfy
\begin{eqnarray*}
\limsup _{n\rightarrow\infty}|P_n(z)|^{1/k_n}\leq 1
\end{eqnarray*}
for all $z\in E$, then
\begin{eqnarray*}
\limsup _{n\rightarrow\infty}|P_n(z)|^{1/k_n}\leq 1
\end{eqnarray*}
for all $z\in \mathbb C^m$. \label{Theo1}\end{theorem} If $E$ is non-thin at $\infty$ then it is easy to see that $V_E\equiv 0$. The converse is not true
for $m\geq 2$, by the following example (which essentially is the same as Example 1.1 in \cite{cegrell-kolodziej-levenberg}).
\begin{example}
Let $E$ be a subset of $\mathbb C^2$ defined by
\begin{eqnarray*}
E=\{(z_1,z_2)\in \mathbb C^2: ~|z_2|\leq 1 \}\bigcup \{(z_1,z_2)\in \mathbb C^2: ~z_1=0\}.
\end{eqnarray*}
Then $V_E\equiv 0$ but $V_E$ is thin at $\infty$. \label{Exam1}\end{example}
\begin{proof}
Use arguments in Example 1.1 in \cite{cegrell-kolodziej-levenberg}, it is easy to see that $V_E\equiv 0$.

If $E$ was non-thin at $\infty$ then since $A=\{(z_1,z_2)\in \mathbb C^2: ~z_1=0\}$ is pluripolar, by Lemma \ref{Lem1} we should have $E\backslash A$ is
non-thin at $\infty$. In particular we should have $V_{E\backslash A}\equiv 0$. However as computed in \cite{cegrell-kolodziej-levenberg} we have
$V_{E\backslash A}(z)=\log ^+|z|$, which is a contradiction.

Hence $E$ is thin at $\infty$.
\end{proof}
The following result gives a characterization of sets $E$ with $V_E\equiv 0$, which also shows that the non-thin at $\infty$ sets are not rare.
\begin{theorem}
Let $E$ be a closed subset of $\mathbb C^m$. Then $V_E\equiv 0$ iff any open neighborhood of $E$ is non-thin at $\infty$. \label{Theo2}\end{theorem} In
\cite{mul-yav}, the authors applied the result of Proposition \ref{Pro1} to various problems of convergence of sequences of polynomials in one complex
variable. In \cite{trong-tuyen} we extended the results to sequences of entire functions of genus zero of one complex variable. With our definition of
non-thinness at $\infty$ for a set $E\subset \mathbb C^m$ here, we can extend Theorems 1 and 3 in \cite{trong-tuyen} to the case of entire functions of
genus zero of several complex variables, under appropriate assumptions.

If $P(z)$ is an entire function of genus zero of one complex variable then we have a representation of $P$ by (see \cite{lev})
\begin{equation}
P(z)=az^\alpha \prod _j(1-\frac{z}{z_j}).\label{Chap0.1}
\end{equation}
Moreover
\begin{eqnarray*}
\sum _{j}\frac{1}{|z_j|}<\infty .
\end{eqnarray*}

In higher dimensions we still have a canonical representation of any entire function of genus zero $P(z_1,\ldots ,z_m)$ (see Proposition 3.16 in
\cite{lelong-gruman}, see also \cite{ronkin}). However in this paper we will use instead the reduction to one variable complex as described below.

Let $P(z_1,\ldots ,z_m)$ be an entire function of genus zero in $\mathbb C ^m$ . Then for any $\lambda \in \mathbb P^{m-1}$ (here $\mathbb P^{m-1}$ means
the complex projective space of dimension $m-1$) which represents a pluricomplex direction in $\mathbb C^m$, we define a one complex variable function
\begin{eqnarray*}
P_{\lambda}(w)=P(w\lambda )
\end{eqnarray*}
where $w\in \mathbb C$. Here we understand the expression $w\lambda$ as follows: $w\lambda :=w\widetilde{\lambda}\in \mathbb C^m $ where
$\widetilde{\lambda}\in S^{2m-1}=\{z\in \mathbb C^m:|z|=1\}$ is a fixed representative of $\lambda$. Then $P_{\lambda}(w)$ is also an entire function of
genus zero of one complex variable hence we can use (\ref{Chap0.1}) to write
\begin{eqnarray*}
P_{\lambda}(w)=a_{\lambda}w^{\alpha_{\lambda}}\prod _{j}(1-\frac{w}{w_{\lambda ,j}}).
\end{eqnarray*}
We define the counting functions of $P_{\lambda}$ and $P$ as follows: If $t>0$ then
\begin{eqnarray*}
\eta _{P}(t,\lambda )&=&\mbox{the number of elements of the set}\{w\in \CC :P_{\lambda }(w)=0,~|w|\leq t\},\\
\eta _P(t)&=&\int _{\mathbb P^{m-1}}n_{P}(t,\lambda )d\lambda
\end{eqnarray*}
where we integrate using the standard metric on $\mathbb P^{m-1}$.

The following result generalizes Theorem 1 in \cite{trong-tuyen}
\begin{theorem}
Let $E\subset \mathbb C^m$ be non-thin at $\infty$.

Let $P_n:\mathbb C^m\rightarrow \mathbb C$ be a sequence of entire functions of genus zero, and let $(k_n)$ be a sequence of positive numbers greater
than $1$. For $\lambda \in \mathbb P^{m-1}$ let us define the function $P_{n,\lambda }:\mathbb C\rightarrow \mathbb C$ by
\begin{eqnarray*}
P_{n,\lambda }(w)=P_n(w\lambda ).
\end{eqnarray*}
Let
\begin{eqnarray*}
P_{n,\lambda}(w)=a_{n,\lambda}w^{\alpha _{n,\lambda}}\prod _{j}(1-\frac{w}{w_{n,\lambda ,j}}),
\end{eqnarray*}
its representation, where
\begin{eqnarray*}
\sum _{j}\frac{1}{|w_{n,\lambda ,j}|}<\infty .
\end{eqnarray*}
Let $\eta _{P_n}(t,\lambda )$ and $\eta _{P_n}(t)$ be counting functions of $P_{n}$ as defined above.

Assume that for any compact set $K\subset \mathbb C^m$ there exists $N_0$ such that the set
\begin{equation}
\{|P_n(z)|^{1/k_n}:~n\geq N_0, z\in K\}\label{Theo4.0}
\end{equation}
is bounded from above.

For each $\lambda \in \mathbb P^{m-1}$ assume that the following three conditions are satisfied:

\begin{equation}
\limsup _{n\rightarrow\infty}\frac{[\eta _{P_{n }}(1,\lambda )+\sum _{|w_{n,\lambda ,j}|\geq 1}1/|w_{n,\lambda ,j}|]}{k_n}<\infty .\label{Theo4.1}
\end{equation}

\begin{equation}
\lim _{R\rightarrow\infty}\limsup_{n\rightarrow\infty}\frac{|\sum _{|w_{n,\lambda ,j}|\geq R}1/w_{n,\lambda ,j}|}{k_n}=0.\label{Theo4.2}
\end{equation}

There exists a sequence $R_{n,\lambda }\rightarrow \infty$ such that
\begin{equation}
\limsup _{n\rightarrow\infty}\frac{\eta _{P_{n}}(R_{n,\lambda },\lambda )}{k_n}\leq \kappa <\infty ,\label{Theo4.3}
\end{equation}
where $\kappa $ is independent of $\lambda$.

Then the following conclusion is true: If
\begin{eqnarray*}
\limsup _{n\rightarrow\infty}|P_n(z)|^{1/k_n}\leq 1
\end{eqnarray*}
for all $z\in E$, then we have
\begin{eqnarray*}
\limsup _{n\rightarrow\infty}|P_n(z)|^{1/k_n}\leq 1
\end{eqnarray*}
for all $z\in \mathbb C^m$. \label{Theo4}\end{theorem} We mention here some remarks about the result of Theorem \ref{Theo4}:

1. Theorem \ref{Theo4} is not a direct consequence of the one-dimensional result in \cite{trong-tuyen}. This is because although $E$ is non-thin at
$\infty$ in $\mathbb C^m$, we may have that $E\cap L$ is thin at $\infty$ in $L$ (even $E\cap L$ may be the empty set) where $L$ is a complex line in
$\mathbb C^m$.

2. The way of considering the growth of the sequence $(P_n)$ in $\mathbb C^m$ by considering the growth of the restricted sequence to any line
$L_{\lambda}$ is a natural approach.

3. Condition (\ref{Theo4.0}) is natural. By Theorem 2 in \cite{trong-tuyen} conditions (\ref{Theo4.2}) and (\ref{Theo4.3}) (with $\kappa $ may depend on
$\lambda$) in Theorem \ref{Theo4} are satisfied in case we have condition (\ref{Theo4.1}),
\begin{eqnarray*}
\limsup _{n\rightarrow\infty}|P_n(z)|^{1/k_n}\leq 1,
\end{eqnarray*}
for all $z\in \mathbb C^m$, and
\begin{eqnarray*}
\liminf _{n\rightarrow\infty}|P_n(0)|^{1/k_n}>0.
\end{eqnarray*}

Similarly, we can extend Theorems 3 in \cite{trong-tuyen} by the same manner as that of Theorem \ref{Theo4}. Because the statement of the result is
rather long, we defer the statement of this result to Section 3 (see Theorem \ref{Theo5}). Here we outline the notations and results we will use in the
statement and the proof of Theorem \ref{Theo5}.

Let $K\subset \mathbb C^m$ be compact. Then the Robin constant of $K$ is defined as (see \cite{levenberg})
\begin{eqnarray*}
\gamma (K)=\limsup _{z\rightarrow \infty}[V_E^*(z)-\log |z|],
\end{eqnarray*}
and the $\mathbb C^m$-capacity of $K$ is
\begin{equation}
C(K)=e^{-\gamma (K)}.\label{Chap0.2}
\end{equation}
Let $K\subset \{z\in \mathbb C^m:~|z|\leq s \}$ be compact and non-pluripolar, where $s>0$. By Theorem 2 in \cite{taylor} and Poisson integration formula
for harmonic functions (see Theorem 2.2.3 in \cite{klimek}), there exists a constant $C_m>0$ depending only on the dimension $m$ such that
\begin{equation}
\sup _{|z|\leq s}V_{K}^*(z)\leq C_m(\log s+\log 2- \log C(K)).\label{taylortheorem}
\end{equation}
The following corollary is an analog of Remark 1 in \cite{mul-yav}
\begin{corollary}
Let $E\subset \mathbb C^m$ be closed. Let $0<C_m< +\infty$ be the constant in (\ref{taylortheorem}). Then $E$ is non-thin at $\infty$ if it satisfies the
following condition: There exists $\beta >0$ such that
\begin{eqnarray*}
\limsup _{R\rightarrow\infty}\frac{\log C(E_R)}{\log R}\geq\beta >\frac{C_m-1}{C_m},
\end{eqnarray*}
where $C(E_R)$ is determined from formula (\ref{Chap0.2}).
\end{corollary}

Interesting questions may arise such as whether we can in fact get rid of the condition (\ref{Theo4.0}) or get rid of the constant $\kappa$ in
(\ref{Theo4.3}) so we have complete analogs with the results in \cite{trong-tuyen}, or whether we have a Wienner criterion for non-thin at $\infty$ in
$\mathbb C^m$...  We hope to return to these issues in a future paper.

This paper consists of three sections. In Section 2 we prove Theorems \ref{Theo1} and \ref{Theo2}, as well as some other properties of sets non-thin at
$\infty$. In Section 3 we prove Theorem \ref{Theo4} as well as the analog of Theorem 3 in \cite{trong-tuyen}.
\section{Non-thin at $\infty$ sets in $\mathbb C^m$}
In this section we explore some properties of sets which are non-thin at $\infty$ in $\mathbb C^m$.
\begin{lemma}
a) Let $E$ be a subset of $\mathbb C^m$. If $E$ is non-thin at $\infty$ and $A$ is pluripolar then $E\backslash A$ is non-thin at $\infty$.

b) Let $E\subset \mathbb C^m$ and $F\subset \mathbb C^n$. Then $E\times F\subset \mathbb C^m\times \mathbb C^n$ is non-thin at $\infty$ iff $E\subset
\mathbb C^m$ and $F\subset \mathbb C^n$ are non-thin at $\infty$. \label{Lem1}\end{lemma}
\begin{proof}
a) Define $F=E\backslash A$. For any $R>0$ the set $E_R=E\cap \{z\in \mathbb C^m :|z|\leq R\}$ is bounded. Hence by Corollary 5.2.5 in \cite{klimek} we
have
\begin{eqnarray*}
V^*_{E_R}=V^*_{F_R}.
\end{eqnarray*}
Fix a sequence $R_n\rightarrow\infty$, using the same argument as in the proof of Theorem \ref{Theo1} there exists a pluripolar set $B$ such that
\begin{eqnarray*}
\lim _{n\rightarrow\infty}V_{F_{R_n}}(z)=\lim _{n\rightarrow\infty}V^*_{F_{R_n}}(z)=\lim _{n\rightarrow\infty}V^*_{E_{R_n}}(z)=\lim
_{R\rightarrow\infty}V_{E_{R_n}}(z)=0.
\end{eqnarray*}
for all $z\in \mathbb C^m\backslash B$. Then a little more argument completes the proof of this part.

b) Let $R>0$. Then one can check easily that
\begin{eqnarray*}
\max\{V_{E_R}(z),V_{F_R}(w)\}\leq V_{E_R\times F_R}(z,w)\leq V_{E_R}(z)+V_{F_R}(w).
\end{eqnarray*}
where $z\in \mathbb C^m,~w\in \mathbb C^n$. This completes the proof of case b).
\end{proof}
Now we proceed to proving Theorem \ref{Theo1}.
\begin{proof}
In this proof fix a sequence $R_n\nearrow \infty$.

($2\Rightarrow 1$) Since $E$ is closed, for any $R>0$ we have $E_R$ is compact.

Fixed $z_0\in \mathbb C^m$. By Siciak's theorem (see Theorem 5.1.7 in \cite{klimek}) there exists a sequence of polynomials $(P_n)$ of degree $(k_n)$
such that
\begin{eqnarray*}
||P_n||^{1/k_n}_{E_{R_n}}\leq 1,
\end{eqnarray*}
for all $n=1,2,\ldots$, and
\begin{equation}
\lim _{n\rightarrow\infty }\log |P_n(z_0)|^{1/k_n}=\lim _{n\rightarrow\infty}V_{E_{R_n}}(z_0).\label{Theo1.1}
\end{equation}
Then for any $z\in E$ we have
\begin{eqnarray*}
\limsup _{n\rightarrow\infty}|P_n(z)|^{1/k_n}\leq 1.
\end{eqnarray*}
Hence by assumption that Statement 2 is true, we have
\begin{eqnarray*}
\limsup _{n\rightarrow\infty}|P_n(z)|^{1/k_n}\leq 1
\end{eqnarray*}
for all $z\in \mathbb C^m$. In particular, with $z=z_0$ we get from (\ref{Theo1.1}) that
\begin{eqnarray*}
\lim _{n\rightarrow\infty}V_{E_{R_n}}(z_0)\leq 0.
\end{eqnarray*}
Since $z_0\in \mathbb C^m$ is arbitrary and obviously $V_{E}(z)\geq 0$ for all $z$, we get Statement 1.

($1\Rightarrow 2$) We use the ideas in \cite{mul-yav}. Assume that Statement 1 is true. Consider any sequence $(P_n)$ of polynomials and $k_n\geq
deg(P_n)$ such that
\begin{equation}
\limsup _{n\rightarrow\infty}|P_n(z)|^{1/k_n}\leq 1\label{Theo1.2}
\end{equation}
for all $z\in E$. For each $n$ define
\begin{eqnarray*}
v_n(z)=\log |P_n(z)|^{1/k_n}.
\end{eqnarray*}

Note that Statement 1 implies that for $R$ large enough then $E_R$ is non-pluripolar. Fix $R>0$ such that $E_R$ is non-pluripolar.

For any $h,l\in \NN$ define
\begin{eqnarray*}
E_R^{h,l}=\bigcap _{n=h}^{\infty}\{z\in E_R: ~v_n(z)<\frac{1}{l}\}.
\end{eqnarray*}
Then $E_{R}^{h,l}\subset E_{R}^{h+1,l}$ and from (\ref{Theo1.2})
\begin{equation}
\bigcup _{h=1}^{\infty}E_R^{h.l}=E_R\label{Theo1.3}
\end{equation}
for any $l\in \NN$.

By definition of the pluricomplex Green function, for any $h,l\in \NN,~n\geq h$ and $z\in \mathbb C^m$
\begin{eqnarray*}
v_n(z)\leq V_{E^{h,l}_R}(z)+\frac{1}{l}\leq V^*_{E^{h,l}_R}(z)+\frac{1}{l}.
\end{eqnarray*}
Hence take the limsup of the above inequality as $n\rightarrow\infty$ we get
\begin{eqnarray*}
v(z)=\limsup _{n\rightarrow\infty}v_n(z)\leq V_{E^{h,l}_R}(z)+\frac{1}{l}\leq V^*_{E^{h,l}_R}(z)+\frac{1}{l},
\end{eqnarray*}
for all $h,l\in \NN,~z\in\mathbb C^m$. Take the limit of this inequality as $h\rightarrow\infty$, using (\ref{Theo1.2}), we see from Corollary 5.2.6 in
\cite{klimek} that
\begin{eqnarray*}
v(z)\leq \lim _{h\rightarrow\infty}V^*_{E^{h,l}_R}(z)+\frac{1}{l}=V^*_{E_R}(z)+\frac{1}{l},
\end{eqnarray*}
for all $l\in \NN,~z\in \mathbb C^m,~R>0$. Since $l$ is arbitrary, we get
\begin{equation}
v(z)\leq V^*_{E_R}(z),\label{Theo1.4}
\end{equation}
for all $z\in\mathbb C^m , ~R>0$.

For each $n\in \NN$ define
\begin{eqnarray*}
A_n=\{z\in \mathbb C^m:~V_{E_{R_n}}(z)<V^*_{E_{R_n}}(z)\}
\end{eqnarray*}
then $A_n$ is pluripolar thus has Lebesgue measure zero, and $V_{E_{R_n}}(z)=V^*_{E_{R_n}}(z)$ for $z\in \mathbb C^m\backslash A_n$. Hence
\begin{eqnarray*}
A=\bigcup _{n=1}^{\infty}A_n
\end{eqnarray*}
is also pluripolar and of Lebesgue measure zero, and we have $V_{E_{R_n}}(z)=V^*_{E_{R_n}}(z)$ for $z\in \mathbb C^m\backslash A$ and $n\in \NN$. Hence
for $z\in \mathbb C^m\backslash A$, apply (\ref{Theo1.4}) we get
\begin{eqnarray*}
v(z)\leq \lim _{n\rightarrow\infty}V^*_{E_{R_n}}(z)=\lim _{n\rightarrow\infty}V_{E_{R_n}}(z)=0
\end{eqnarray*}
for all $z\in\mathbb C^m\backslash A$. Since $A$ is of Lebesgue measure zero, by definition of $v(z)$, we see that $v(z)\leq 0$ for all $z\in \mathbb
C^m$. This completes the proof of Theorem \ref{Theo1}.
\end{proof}
Now we prove Theorem \ref{Theo2}.
\begin{proof}

($\Rightarrow$) Let $E$ be a closed subset of $\mathbb C^m$. Assume that $V_E\equiv 0$. We will show that any open neighborhood of $E$ is non-thin at
$\infty$. Let $F$ be any open neighborhood of $E$. Then for any $R>0$ since $E_R$ is compact, $F$ is open and contains $E_R$, we can find $1>\epsilon
=\epsilon _R>0$ such that
\begin{eqnarray*}
E_{R,\epsilon}=\{z\in \mathbb C^m:~dist(z,E_R)\leq \epsilon \}\subset F.
\end{eqnarray*}
By Corollary 5.1.5 in \cite{klimek} we have $V^*_{E_{R,\epsilon}}(z)= 0$ for $z\in E_{R,\epsilon}\supset E_R$. Now it is obvious that
$E_{R,\epsilon}\subset F_{R+1}$ hence for $z\in E_R$
\begin{equation}
V_{F_{R+1}}^*(z)=0.\label{Theo2.1}
\end{equation}
Define
\begin{eqnarray*}
v(z)=\lim _{R\rightarrow\infty}V^*_{F_R}(z),
\end{eqnarray*}
then $v$ is the limit of a decreasing sequence of PSH functions hence $v$ is itself PSH. By (\ref{Theo2.1}) for $z\in E$ we have $v(z)\equiv 0$. Thus by
definition of the pluricomplex Green function
\begin{eqnarray*}
v(z)\leq V_E(z)=0
\end{eqnarray*}
for all $z\in \mathbb C^m$. This shows that $F$ is non-thin at $\infty$.

($\Leftarrow$) Assume that $E$ is an arbitrary subset of $\mathbb C^m$ with $V_E\not\equiv 0$. Then $V_E(z_0)>0$ for some $z_0\in \mathbb C^m$, hence by
definition of the pluricomplex Green function, there exist a function $u\in L$ such that $u(z)\leq 0$ for $z\in E$, and $u(z_0)>0$. Define
\begin{eqnarray*}
F=\{z\in \mathbb C^m:~u(z)<u(z_0)/2\}.
\end{eqnarray*}
$F$ is open because $u$ is upper-semicontinuous, and $E\subset F$ because $u|_E\leq 0$ and $u(z_0)>0$. Now $u(z)<u(z_0)/2$ for $z\in F$, hence we have
\begin{eqnarray*}
u(z)\leq V_F(z)+u(z_0)/2
\end{eqnarray*}
for all $z\in \mathbb C^m$. In particular choose $z=z_0$ we have $V_F(z_0)\geq u(z_0)/2 >0$, hence $F$ is thin at $\infty$.
\end{proof}
Theorem \ref{Theo2} applied to Example \ref{Exam1} shows that any nonempty open cone in $\mathbb C^m$ is non-thin at $\infty$. The following result
provide other sources of sets which are non-thin at $\infty$.
\begin{remark}
Let $\Delta$ be a collection of complex lines $L$ in $\mathbb C^m$ such that
\begin{eqnarray*}
\bigcup _{L\in \Delta}L=\mathbb C^m.
\end{eqnarray*}
Let $E$ be a closed subset of $\mathbb C^m$ such that for each $L\in \Delta$ the set $E\cap L$ considered as a subset in the one-dimensional complex line
$L$ is non-thin at $\infty$. Then $E$ is non-thin at $\infty$ as a subset in $\mathbb C^m$. \label{Lem2}\end{remark}
\begin{proof}
By Theorem \ref{Theo1} we need only to show that: If $(P_n)$ is a sequence of polynomials, and $k_n\geq deg(P_n)$ such that
\begin{eqnarray*}
\limsup _{n\rightarrow\infty}|P_n(z)|^{1/k_n}\leq 1
\end{eqnarray*}
for $z\in E$, then
\begin{eqnarray*}
\limsup _{n\rightarrow\infty}|P_n(z)|^{1/k_n}\leq 1
\end{eqnarray*}
for all $z\in \mathbb C^m$.

Let $w\in \mathbb C$ be a coordinate for $L$ such that the coordinates $z_1,\ldots ,z_n$ of $\mathbb C^m$ are linear functions of $w$ when restricted on
$L$, and denote by $P_{n,L}(w)$ the restriction of $P_n$ to $L$. Then $P_{n,L}(w)$ is a polynomial in one complex variable $w$ of degree
$deg(P_{n,L}(w))\leq deg(P_n)\leq k_n$.

Then since
\begin{eqnarray*}
\limsup _{n\rightarrow\infty}|P_{n,L}(w)|^{1/k_n}\leq 1
\end{eqnarray*}
for $w\in E\cap L$, and $E\cap L$ is non-thin at $\infty$ in the complex line $L$, hence
\begin{eqnarray*}
\limsup _{n\rightarrow\infty}|P_n(w)|^{1/k_n}\leq 1
\end{eqnarray*}
for all $w\in L$. Since this is true for any complex line $L\in \Delta$, it is also true for their union, which is equal to $\mathbb C^m$.
\end{proof}
\section{The growth of sequences of entire functions of genus zero}
In this section we apply the results in previous section to sequences of entire functions of genus zero in $\mathbb C^m$.

First we give a proof of Theorem \ref{Theo4}.
\begin{proof}
Define functions from $\mathbb C^m$ to $[-\infty ,+\infty ]$ by
\begin{eqnarray*}
u(z)=\limsup _{n\rightarrow\infty}u_n(z),
\end{eqnarray*}
where
\begin{eqnarray*}
u_n(z)=\sup _{j\geq n}\frac{1}{k_j}\log |P_{j}(z)|.
\end{eqnarray*}
From the assumption (\ref{Theo4.0}), by Theorem 4.6.3 in \cite{klimek}, $u^*$ is PSH. Moreover there is a pluripolar set $A\subset \mathbb C^m$ such that
$u(z)=u^*(z)$ for all $z\in \mathbb C^m\backslash A$.

Fix $\lambda \in \mathbb P^{m-1}$. Choose $R_n=R_{n,\lambda }$ as the sequence in condition (\ref{Theo4.3}). Then from the proof of Theorem 1 in
\cite{trong-tuyen} we see that for all $w\in \mathbb C$
\begin{eqnarray*}
\limsup _{n\rightarrow\infty}u_n(w\lambda )=\limsup _{n\rightarrow\infty}\frac{1}{k_n}\log |Q_{n,\lambda }(w)|
\end{eqnarray*}
where $Q_{n,\lambda}:\mathbb C\rightarrow \mathbb C$ is defined by

\begin{eqnarray*}
Q_{n,\lambda}(w)=a_{n,\lambda}w^{\alpha _{n,\lambda}}\prod _{|w_{n,\lambda ,j}|\leq R_n}(1-\frac{w}{w_{n,\lambda ,j}}).
\end{eqnarray*}

Condition (\ref{Theo4.0}) shows that
\begin{eqnarray*}
\limsup _{n\rightarrow\infty}\frac{1}{k_n}\log |Q_{n,\lambda}(w)|\leq \kappa \log C_0
\end{eqnarray*}
for all $w\in \mathbb C$ with $|w|\leq \epsilon$. This, together with condition (\ref{Theo4.3}) and potential theory in one complex variable (Berstein's
lemma for polynomials, see also the proof of Theorem \ref{Theo1}), show that there exists $C>0$ such that
\begin{equation}
\limsup _{n\rightarrow\infty}\frac{1}{k_n}\log |Q_{n,\lambda}(w)|\leq \kappa \log ^+|w|+C
\end{equation}
for all $w\in \mathbb C$. Moreover $C$ is independent of $\lambda$.

Hence $u^*/\kappa $ is in $\mathbb L$. Let $F=E\backslash A$, then by Lemma \ref{Lem1} $F$ is non-thin at $\infty$. In particular $V_F\equiv 0$. Now for
all $z\in F$ we have by assumption
\begin{eqnarray*}
u^*(z)\leq 0,
\end{eqnarray*}
thus
\begin{eqnarray*}
u^*(z)\leq \kappa V_F(z)=0
\end{eqnarray*}
for all $z\in \mathbb C^m$.
\end{proof}
Now we consider the analog of Theorem 3 in \cite{trong-tuyen}. We have the following result.
\begin{theorem}
Let $E\subset \mathbb C^m$ be closed and satisfy the following condition: there exists $\beta >0$ such that
\begin{equation}
\limsup _{R\rightarrow\infty}\frac{\log C(E_R)}{\log R}\geq\beta >\frac{C_m-1}{C_m},\label{Theo5.1}
\end{equation}
where $C(E_R)$ is determined from formula (\ref{Chap0.2}), and $0<C_m<\infty$ is the constant in (\ref{taylortheorem}).

Let $(P_n)$ and $(k_n)$ be sequences satisfying conditions (\ref{Theo4.0}), (\ref{Theo4.1}), (\ref{Theo4.2}) and (\ref{Theo4.3}) of Theorem \ref{Theo4}.
Then for $\lambda \in \mathbb P^{m-1}$ we have
\begin{eqnarray*}
\exp \{\limsup _{n\rightarrow\infty}\frac{1}{2\pi k_n}\int _0^{2\pi}\log |P_n(e^{i\theta }\lambda )|d\theta \}= C_{\lambda}<\infty .\label{Theo5.2}
\end{eqnarray*}
Assume that for all $z\in E$ we have
\begin{eqnarray*}
\limsup _{n\rightarrow\infty}|P_n(z)|^{1/k_n}\leq h(|z|)
\end{eqnarray*}
where
\begin{equation}
\limsup _{R\rightarrow\infty}\frac{\log h(R)}{\log R}\leq \tau <\infty .
\end{equation}
Then for any $w\in \mathbb C$ and $\lambda \in \mathbb P^{m-1}$ we have
\begin{eqnarray*}
\limsup _{n\rightarrow\infty}|P_n(w\lambda )|^{1/k_n}\leq C_{\lambda}(1+|w|)^{{\tau}/{[1-C_m(1-\beta )]}}.
\end{eqnarray*}
\label{Theo5}\end{theorem}
\begin{proof}
Define functions $u$ and $Q_{n,\lambda}$ as in the proof of Theorem \ref{Theo4}. Without loss of generality we may assume that $u\geq 0$. By part (i) of
Lemma 1 in \cite{trong-tuyen} and the proof of Theorem \ref{Theo4} there exists a constant $C>0$ such that
\begin{equation}
u^*(z)\leq \kappa \log ^+|z|+C,\label{Theo5.3}
\end{equation}
for all $z\in \mathbb C^m$. Now we define
\begin{equation}
\kappa _0=\limsup _{z\in \mathbb C^m,~z\rightarrow\infty}\frac{u^*(z)}{\log |z|}.\label{Theo5.4}
\end{equation}
For any $R>0$, by the definition of the pluricomplex Green function
\begin{equation}
u^*(z)\leq \kappa _0V_{E_R}(z)+h(R).\label{Theo5.5}
\end{equation}
For any $\lambda \in \mathbb P^{m-1}$ define
\begin{eqnarray*}
\kappa (\lambda )=\limsup _{w\in \mathbb C,~w\rightarrow\infty}\frac{u(w\lambda )}{\log |w|}\leq \kappa _0.
\end{eqnarray*}
Fixed $\lambda \in\mathbb P^{m-1}$. By definition
\begin{eqnarray*}
u(w\lambda )=\limsup _{n\rightarrow\infty}\frac{1}{k_n}\log |Q_n(w\lambda )|
\end{eqnarray*}
for all $w\in \mathbb C$. Hence using (\ref{Theo5.5}), for any $s>0$ we have
\begin{eqnarray*}
\limsup _{n\rightarrow\infty}\frac{1}{2\pi}\int _0^{2\pi}\frac{1}{k_n}\log |Q_n(se^{i\theta }\lambda )|d\theta&\leq&\frac{1}{2\pi}\int
_0^{2\pi}u(se^{i\theta }\lambda )d\theta\\
&\leq&\kappa _0\frac{1}{2\pi}\int _0^{2\pi}V_{E_s}(se^{i\theta }\lambda )d\theta +h(s).
\end{eqnarray*}
By the previous inequality, (\ref{taylortheorem}) and assumption (\ref{Theo5.1}) we get
\begin{equation}
\liminf _{s\rightarrow\infty}\limsup _{n\rightarrow\infty}\frac{1}{\log s}\frac{1}{2\pi}\int _0^{2\pi}\frac{1}{k_n}\log |Q_n(se^{i\theta }\lambda
)|d\theta\leq \kappa _0C_m(1-\beta )+\tau .
\end{equation}
Then by Lemma 2 in \cite{trong-tuyen} we have
\begin{equation}
\kappa (\lambda )\leq \kappa _0C_m(1-\beta )+\tau .\label{Theo5.6}
\end{equation}
Since (\ref{Theo5.6}) is true for any $\lambda \in \mathbb P^{m-1}$, by definition (\ref{Theo5.4}) and Lemma 2 in \cite{trong-tuyen} we have
\begin{eqnarray*}
\kappa _0\leq \kappa _0C_m(1-\beta )+\tau ,
\end{eqnarray*}
or equivalently
\begin{eqnarray*}
\kappa _0\leq \frac{\tau }{1-C_m(1-\beta )}.
\end{eqnarray*}
From the above inequality, use Lemma 2 in \cite{trong-tuyen} we get the conclusion of Theorem \ref{Theo5}.
\end{proof}
% Set the ending of a LaTeX document


\begin{thebibliography}{xx}

\bibitem{bedford-taylor1}{Eric Bedford and B. A. Taylor,}  \textit{Plurisubharmonic functions with logarithmic singularities,} Annales de Int. Fourier (Grenoble) 38
(1988), 133-171.

\bibitem{bedford-taylor2}{Eric Bedford and B. A. Taylor,}  \textit{A new capacity for plurisubharmonic functions,} Acta Math. 149 (1982), No 1-2, 1--40.

\bibitem{cegrell-kolodziej-levenberg}{U, Cegrell, S. Kolodziej and N. Levenberg,}  \textit{Two problems on potential theory for unbounded sets,} Math. Scan. 83 (1998),
265-276.

\bibitem{trong-tuyen}{Dang Duc Trong and Truong Trung Tuyen,}  \textit{The growth of entire functions of genus zero,} arXiv: 0610045.

\bibitem{klimek}{Maciej Klimek,}  \textit{Plutipotential theory,} Courier International Ltd., Scotland, 1991.

\bibitem{Kolodziej}{S. Kolodziej,}  \textit{The logarithmic capacity in $\mathbb C^n$,} Annales Polonici Mathematici, XLVIII (1988).

\bibitem{lelong-gruman}{P. Lelong and L. Gruman,}  \textit{Entire functions of several complex variables,} Springer-Verlag ,Berlin Heidelberg, 1986.

\bibitem{levenberg}{Norman Levenberg,}  \textit{Capacities in several complex variables,} PhD thesis, The University of Michigan, 1984.

\bibitem{lev}{B. Ya. Levin,} \textit{Lectures on Entire Functions,} Translations of Mathematical Monographs, Vol. 150, AMS, Providence, Rhode Island, 1996.

\bibitem{mul-yav}{J. Muller and A. Yavrian,} \textit{On polynomials sequences with restricted growth near infinity,} Bull. London Math. Soc. 34 (2002), pp. 189-199.

\bibitem{ransford}{Thosmas Ransford,} \textit{Potential theory in the complex plane,} London Maths. Soc. Student Texts 28, Cambridge University Press, 1995.

\bibitem{ronkin}{L. I. Ronkin,}  \textit{Introduction to the theory of entire functions of several variables,} Trans. Maths. Monographs (Vol 44), AMS 1974.

\bibitem{siciak}{Josef Siciak,}  \textit{Extremal plurisubharmonc functions in $\mathbb C^n$,} Annales Polonici Mathematici, XXXIX (1981).

\bibitem{taylor}{B. A. Taylor,}  \textit{An estimate for an extremal plurisubharmonic function $\mathbb C^n$,} Seminar P. Lelong-P. Dolbeaut-H. Skoda,
Lecture Notes in Mathematics 1028, Springer-Verlag, Belin Heidelberg New York Tokyo, 1983.
\end{thebibliography}
\end{document}